\documentclass[reqno]{amsart}

\usepackage{bbm}
\usepackage[top=2.65cm,left=2.8cm,right=2.8cm,bottom=2.65cm]{geometry}
\usepackage{amsfonts}
\usepackage{amsmath,amsthm,amscd,mathrsfs,amssymb,graphicx,enumerate,
	dsfont}
\usepackage{xypic}
\usepackage{hyperref}
\usepackage[usenames,dvipsnames]{xcolor}
\hypersetup{colorlinks=true,citecolor=NavyBlue,linkcolor=Brown,urlcolor=Orange}

\allowdisplaybreaks

\newtheorem{theorem}{Theorem}[section]
\newtheorem{voisinconjecture}[theorem]{Voisin's Conjecture}
\newtheorem{noriconjecture}[theorem]{Nori's Conjecture}
\newtheorem{mainconjecture}[theorem]{Main Conjecture}
\newtheorem{lemma}[theorem]{Lemma}

\newtheorem{conjecture}[theorem]{Conjecture}
\newtheorem{proposition}[theorem]{Proposition}
\newtheorem{prop}[theorem]{Proposition}

\theoremstyle{definition}
\newtheorem{remark}[theorem]{Remark}

\newtheorem{example}[theorem]{Example}
\newtheorem{examples}[theorem]{Examples}
\newtheorem{definition}[theorem]{Definition}

\newtheoremstyle{TheoremNum}
        { 1em}{1 em}              
        {\itshape}                      
        {}                              
        {\bfseries}                     
        {.}                             
        { 0.5 em}                             
        {\thmname{#1}\thmnote{ \bfseries #3}}
    \theoremstyle{TheoremNum}

    \newtheoremstyle{TheoremNum}
        { 1em}{1 em}              
        {\itshape}                      
        {}                              
        {\bfseries}                     
        {.}                             
        { 0.5 em}                             
        {\thmname{#1}\thmnote{ \bfseries #3}}
    \theoremstyle{TheoremNum}

 \newtheoremstyle{TheoremNum}
        { 1em}{1 em}              
        {\itshape}                      
        {}                              
        {\bfseries}                     
        {.}                             
        { 0.5 em}                             
        {\thmname{#1}\thmnote{ \bfseries #3}}
    \theoremstyle{TheoremNum}

     \newtheoremstyle{TheoremNum}
        { 1em}{1 em}              
        {\itshape}                      
        {}                              
        {\bfseries}                     
        {.}                             
        { 0.5 em}                             
        {\thmname{#1}\thmnote{ \bfseries #3}}
    \theoremstyle{TheoremNum}

         \newtheoremstyle{TheoremNum}
        { 1em}{1 em}              
        {\itshape}                      
        {}                              
        {\bfseries}                     
        {.}                             
        { 0.5 em}                             
        {\thmname{#1}\thmnote{ \bfseries #3}}
    \theoremstyle{TheoremNum}

\usepackage[latin1]{inputenc}
\usepackage{amsmath,tikz-cd}
\usepackage{setspace}
\usepackage{amsfonts, mathrsfs}
\usepackage{amssymb}

\newcommand{\CH}{\textup{CH}}

\newcommand{\End}{\text{End}}

\newcommand{\Z}{\mathbb{Z}}

\newcommand{\Sym}{\text{Sym}}

\usepackage{pdflscape}

\newcommand{\ho}{\mathfrak{h}^\circ}
\newcommand{\h}{\mathfrak{h}}
\newcommand{\PP}{\mathbb{P}}

\numberwithin{equation}{section}

\begin{document}

\bibliographystyle{alpha.bst}

\title{Effective zero-cycles and the Bloch--Beilinson filtration}

\author{Olivier Martin}
\author{Charles Vial}

\address{Instituto de Matem\'atica Pura e Aplicada, Brasil}
\email{olivier.martin@impa.br}

\address{Universit\"at Bielefeld, Germany}
\email{vial@math.uni-bielefeld.de}

\thanks{The research of Ch.V.\ was funded by the Deutsche Forschungsgemeinschaft (DFG, German Research Foundation) -- Project-ID 491392403 -- TRR 358}

\begin{abstract}
A conjecture of Voisin states that two points on a smooth projective complex variety whose algebra of holomorphic forms is generated in degree 2 are rationally equivalent to each other if and only if their difference lies in the third step of the Bloch--Beilinson filtration. In this note, we formulate a generalization that allows for rational equivalence of effective zero-cycles of higher degree, at the expense of looking deeper in the Bloch--Beilinson filtration. In the first half, we provide evidence in support of this conjecture 
in the case of abelian varieties and projective hyper-K\"ahler manifolds. Notably, we give explicit criteria for rational equivalence of effective zero-cycles on moduli spaces of semistable sheaves on K3 surfaces,  generalizing that of Marian--Zhao. In the second half, in an effort to explain our main conjecture, we formulate a second conjecture predicting when the diagonal of a smooth projective variety belongs to a subalgebra of the ring of correspondences generated in low degree.
\end{abstract}

\maketitle

\section{Introduction}

We work throughout over the field of complex numbers and with Chow rings with rational coefficients. The class in $\CH_0(X)$ of a closed point $x\in X$ will be denoted by $[x]$.

\subsection{Effective zero-cycles on moduli spaces of sheaves on K3 surfaces}
Building on work of O'Grady~\cite{OG-moduli} and Shen--Yin--Zhao~\cite{SYZ}, Marian and Zhao established the following criterion for two points on a smooth projective moduli space of semistable sheaves on a projective K3 surface to be rationally equivalent.

\begin{theorem}[Marian--Zhao \cite{MZ}\,; see Theorem~\ref{MZ}] \label{T2:MZ}
	 Let $\mathcal M$ be a smooth projective moduli space of semistable sheaves on a
	projective K3 surface $S$ and
	let $\mathcal{F}$ and $\mathcal{G}$ be closed points of $\mathcal M$. Then
	\begin{align*}
	[\mathcal{F}]= [\mathcal{G}] \ \mbox{in}\ \CH_0(\mathcal M) 
\	\iff \
  c_2(\mathcal{F}) = c_2(\mathcal{G})\ \mbox{in}\ \CH_0(S).
	\end{align*}
\end{theorem}

A first aim of this paper is to establish the following generalization to effective zero-cycles of arbitrary degree:

\begin{theorem} [Theorem~\ref{T}]\label{T2:T}
		 Let $\mathcal M$ be as in Theorem~\ref{T2:MZ} and let $2n$ be its dimension.
		 Let $\mathcal{F}_1, \ldots, \mathcal{F}_m$ and $\mathcal{G}_1,\ldots, \mathcal{G}_m$ be closed points of $\mathcal M$. Then
	$$\sum_{i=1}^{m}[\mathcal{F}_i]=\sum_{i=1}^{m} [\mathcal{G}_i] \ \mbox{in}\ \CH_0(\mathcal M)  
	\ \iff \
	 \sum_{i=1}^{m} c_2(\mathcal{F}_i)^{\times k} =\sum_{i=1}^{m} c_2(\mathcal{G}_i)^{\times k} \ \mbox{in}\ \CH_0(S^k) \ \mbox{for all } k\leq \min (m,n).$$
\end{theorem}

We also refer to Theorem~\ref{T2} and Theorem~\ref{T3}  for versions of Theorem~\ref{T2:T} concerned with other types of hyper-K\"ahler varieties, namely generalized Kummer varieties and Fano varieties of lines on smooth cubic fourfolds.

\subsection{Rational equivalence of effective zero-cycles}
As outlined in \cite[\S 2.2]{V3},
motivated by the criterion of Marian--Zhao, Voisin formulated the following:

\begin{voisinconjecture}[{\cite[Conj.~2.11]{V3}}]\label{Vconj}
	Let $X$ be a smooth projective variety whose algebra of holomorphic forms is generated in degree $\leq 2$. Then, given closed points $x,y$ of $X$,
	\begin{equation*}
	[x]=[y] \ \mbox{in} \ \CH_0(X) \ \iff \  [x]=[y] \ \mbox{in} \ \CH_0(X)/F_{\mathrm{BB}}^3\CH_0(X),
	\end{equation*}
	where $F_{\mathrm{BB}}^\bullet$ denotes the (conjectural) Bloch--Beilinson filtration.
\end{voisinconjecture}

In turn, motivated by Theorem~\ref{T2:T}, we propose the following generalization of Voisin's Conjecture~\ref{Vconj}, by allowing for rational equivalence of effective zero-cycles of higher degree,  at the expense of looking deeper into the Bloch--Beilinson filtration. 

\begin{mainconjecture}\label{mainconj}
Let $X$ be a smooth projective variety whose algebra of holomorphic forms is generated in degree $\leq d$. Then, for closed point $x_1,\ldots, x_m,y_1,\ldots, y_m\in X,$
\begin{equation*}
\sum_{i=1}^m [x_i]=\sum_{i=1}^m [y_i] \ \mbox{in} \ \CH_0(X) \ \ \iff \ \ \sum_{i=1}^m [x_i]=\sum_{i=1}^m [y_i] \ \mbox{in} \ \CH_0(X)/F_{\mathrm{BB}}^{md+1}\CH_0(X).
\end{equation*}
Equivalently, the image of the map
\begin{align*}
\Phi_m: \ \ \ \ \  \  \textup{Sym}^mX\times \textup{Sym}^mX \ \  \ \ \ \ & \ \longrightarrow \  \ \ \ \CH_0(X)\\
(x_1+\cdots+x_m,y_1+\cdots+y_m)& \ \longmapsto  \ \sum_{i=1}^m [x_i]-\sum_{i=1}^m [y_i].\end{align*}
 intersects $F_{\mathrm{BB}}^{md+1}\CH_0(X)$ only in $\{0\}$.
\end{mainconjecture}

\begin{remark}\label{triv}
	Since the conjectural Bloch--Beilinson filtration satisfies $F_{\mathrm{BB}}^r\CH_0(X)=0$ for $r\geq \dim X$, Main Conjecture \ref{mainconj} is only interesting for small $m$ and $d$. 
	For instance, if the Bloch--Beilinson filtration exists, the conjecture holds trivially for curves or varieties with an indecomposable top form e.g., Calabi-Yau varieties, complete intersections of general type, etc.
\end{remark}

In practice, we will consider Main Conjecture~\ref{mainconj} with respect to candidate filtrations for the Bloch--Beilinson filtration.
As a first example, using Beauville's filtration $F_{\mathrm{B}}^\bullet$ as the candidate Bloch--Beilinson filtration for abelian varieties, we have:
\begin{theorem}[Theorem \ref{uncondabvar}]
Let $A$ be an abelian $g$-fold and let $x_1,\ldots, x_m,y_1,\ldots, y_m\in A$,
\begin{align*}
 \sum_{i=1}^m [x_i]=\sum_{i=1}^m [y_i]   \ \mbox{in} \ \textup{CH}_0(A) 
 \ \iff \
 \sum_{i=1}^m [x_i] = \sum_{i=1}^m [y_i]  \ \mbox{in} \ 
 \CH_0(X)/F_{\mathrm{B}}^{m+1}\CH_0(X).
\end{align*}
\end{theorem}
Here, we write 
$$ F_{\mathrm{B}}^{m+1}\CH_0(X) =_{\textup{def}} \bigoplus_{s=m+1}^g \textup{CH}^g_{(s)}(A),$$ where $\textup{CH}^g_{(s)}(A)$ are the graded pieces of the Beauville filtration on $\textup{CH}_0(A)$:
$$\textup{CH}^g_{(s)}(A) \ =_{\textup{def}} \ \{\alpha \in \textup{CH}^g(A): [k]^*(\alpha)=k^{2g-s}\alpha \text{ for all } k\in \mathbb{Z} \},$$
where $[k]: A\longrightarrow A$ is the multiplication by $k$ isogeny.
\medskip

There are two further candidates for the Bloch--Beilinson filtration on $\CH_0(X)$ that we will consider in this work. The first, $F_{\mathrm{V}}^\bullet\CH_0(X)$, was proposed by Voisin in \cite{V4}, and is defined by
$$F^i_{\mathrm{V}}\CH_0(X) \ = \ \bigcap_{\Gamma, Y}\ker\big(  \Gamma_*: \CH_0(X)\longrightarrow \CH_0(Y) \big),$$
where $Y$ ranges over all smooth projective $(i-1)$-folds and $\Gamma$ over all correspondences in the group $\CH^{i-1}(X\times Y)$.
In Proposition~\ref{abprop} and in Proposition \ref{HKprop}, we adapt arguments of Voisin from \cite{V3} to deduce Main Conjecture \ref{mainconj}  for abelian varieties and for hyper-K\"ahler varieties with respect to the candidate filtration $F_{\mathrm{V}}^\bullet$ from well-known conjectures on algebraic cycles.
\medskip

The second candidate, which is particularly useful to provide unconditional evidence in favor of Main Conjecture~\ref{mainconj}, is inspired by Murre's filtration \cite{Murre2} and
uses the language of birational motives as introduced by Kahn--Sujatha \cite{KS}, see \cite[\S 2]{Vial-JMPA} for an overview. Briefly, a birational correspondence between two connected smooth projective varieties $X$ and $Y$ over a field $k$ is a cycle 
\[\gamma \in \lim_{U \subseteq X} \CH^{\dim Y}(U\times_k Y) = \CH_0(Y_{k(X)}),\] where the limit runs through all Zariski open subsets of $X$ and where $k(X)$ is the function field of~$X$. Birational correspondences can be composed and there is a well-defined induced action 
\[\gamma_*: \CH_0(X) \to \CH_0(Y),\]
which in fact determines $\gamma$ if $k$ is a universal domain \cite[Lem.~2.1]{Vial-JMPA}.
If $k$ is the field of complex numbers, then there is, for all integers $i$, a well-defined induced action 
$$\gamma^*\colon H^0(Y,\Omega_Y^i) \to H^0(X,\Omega_X^i).$$ 
A \emph{birational motive} is a pair $(X,\varpi)$, also denoted $\ho_{\varpi}(X)$,
 consisting of a smooth projective variety $X$ and a birational idempotent correspondence $\varpi \in \CH_0(X_{k(X)})$.
 When $\varpi$ is the generic point of $X$, namely the identity birational correspondence, we write simply $\ho(X)$ instead of $\ho_{\varpi}(X)$.
 \medskip

Let now $\varpi \in \End(\ho(X))$ be any birational idempotent correspondence on a smooth projective complex variety $X$ and denote by $\delta^{k-1} : X \hookrightarrow X^k$ the diagonal embedding, where by convention $\delta^{-1}$ is the structure morphism.
These induce a morphism  
$$\bigoplus_{k\geq 0} \varpi^{\otimes k}\circ \delta^{k-1}\colon \ho(X) \longrightarrow \operatorname{Sym}^*\ho_{\varpi}(X).$$
If this morphism is split injective, we say that $\ho(X)$ is \emph{co-generated} by $\ho_{\varpi}(X)$;
in that case, the algebra of global holomorphic forms $$H^0(X,\Omega_X^\bullet) = \bigoplus_{i\geq 0}H^0(X,\Omega_X^i)$$ is generated by the image of $\varpi^*$.
\medskip

If $\ho(X)$ is co-generated by $\ho_{\varpi}(X)$ and if the algebra of global holomorphic forms on $X$ is generated in degree $\leq d$,
 we define, for all $j$ and all $1\leq i \leq d$, the following descending filtration
$$F_\varpi^{dj+i} \CH_0(X) =_{\mathrm{def}} \ker \left( \bigoplus_{k=0}^j \varpi^{\otimes k} \circ \delta^{k-1}_* \ \colon\ \CH_0(X) \longrightarrow \bigoplus_{k=0}^j \CH_0(X^k)\right). $$
This filtration is in general finer than the conjectural Bloch--Beilinson filtration.
However, if $\varpi$ acts on $H^0(X,\Omega_X^{i})$ as zero for $i> d$, then $F_\varpi^{dj+1}\CH_0(X)$ coincides with $F_{\mathrm{BB}}^{dj+1}\CH_0(X)$ for all $j\geq 0$, provided the conjectural Bloch-Beilinson filtration $F_{\mathrm{BB}}^\bullet$ exists.
Moreover, if in addition $H^0(X,\Omega_X^\bullet)$ is generated in degree~$d$,
 then the above filtration is in fact a candidate for the Bloch--Beilinson filtration.
\medskip
	 
In this language, we have
\begin{prop}[Proposition~\ref{P:cogen2}]\label{P2:cogen2}
	Let $X$ be a smooth projective variety whose algebra of holomorphic forms is generated in degree $\leq d$.
  Assume that there exists a direct summand  $\ho_{\varpi}(X) =_{\mathrm{def}} (X,\varpi)$  of the birational motive $\ho(X)$ such that $\ho(X)$ is co-generated by $\ho_{\varpi}(X)$.
	If $x_1, \ldots, x_m$ and $y_1, \ldots , y_m$ are closed points of~$X$, then
	$$\sum_{i=1}^{m}[x_i]=\sum_{i=1}^{m} [y_i] \ \mbox{in}\ \CH_0(X) \ \iff \ \sum_{i=1}^{m}[x_i]=\sum_{i=1}^{m} [y_i] \ \mbox{in}\ \CH_0(X)/F_\varpi^{dm+1}\CH_0(X). $$
\end{prop}

The link to Main Conjecture \ref{mainconj} is provided by the following.
Assume that the algebra $H^0(X,\Omega_X^\bullet) $ is generated by $\bigoplus_{i\leq d}H^0(X,\Omega_X^i)$.  
A combination of the standard conjectures and of the Bloch--Beilinson conjecture implies the existence of a birational idempotent correspondence $\varpi$ such that $\varpi_* H^0(X,\Omega_X^\bullet) =   \bigoplus_{i\leq d}H^0(X,\Omega_X^i)$ and, for any choice of such $\varpi$, 
	$\ho(X)$ is co-generated by $\ho_{\varpi}(X)$ (this is \cite[Conj.~5.1]{Vial-JMPA}).
	Hence, provided the conjectural Bloch-Beilinson filtration $F_{\mathrm{BB}}^\bullet$ exists, we have that $F_\varpi^{dj+1}\CH_0(X)$ coincides with $F_{\mathrm{BB}}^{dj+1}\CH_0(X)$ for all $j\geq 0$, and one may thus formulate the following variant of Main Conjecture \ref{mainconj}:
	
	\begin{conjecture}\label{C}
		Let $X$ be a smooth projective variety whose algebra of holomorphic forms is generated in degree~$\leq d$. Then there exists a birational idempotent correspondence $\varpi$ such that  $\varpi_* H^0(X,\Omega_X^\bullet) = \bigoplus_{i\leq d}  H^0(X,\Omega_X^i)$ and such that, for
		$x_1, \ldots, x_m$ and $y_1, \ldots , y_m$ closed points on $X$, 
		$$\sum_{i=1}^{m}[x_i]=\sum_{i=1}^{m} [y_i] \ \mbox{in}\ \CH_0(X) 
		\ \iff \
		 \sum_{i=1}^{m}[x_i]=\sum_{i=1}^{m} [y_i] \ \mbox{in}\ \CH_0(X)/F_\varpi^{dm+1}\CH_0(X). $$ 
	\end{conjecture}

Recall that the algebra of global holomorphic forms on a hyper-K\"ahler variety is generated in degree~2.
The following gives evidence for Conjecture~\ref{C} and therefore Main Conjecture \ref{mainconj} in the case of hyper-K\"ahler varieties.
	
	\begin{theorem}[Theorem~\ref{HKevidence}]
		Let $X$ be a hyper-K\"ahler variety. Assume that $X$ is one of the following:
		\begin{enumerate}
			\item \label{hilb} $\mathrm{Hilb}^n(S)$, the Hilbert scheme of length-$n$ closed subschemes on a K3 surface $S$\,;
			\item \label{moduli} $\mathrm{M}_\sigma(v)$, a moduli space of stable objects on a K3 surface\,;
			\item \label{kum} $K_n(A)$, the generalized Kummer variety associated to an abelian surface $A$\,;
			\item \label{fano} $F(Y)$, the Fano variety of lines on a smooth cubic fourfold $Y$\,;
			\item \label{OG6} $\widetilde{K}_v(A)$, O'Grady's six-dimensional example.
		\end{enumerate}
	Then there exists a birational idempotent correspondence $\varpi$ such that $\ho(X)$ is co-generated by $\ho_{\varpi}(X)$,   $\varpi_* H^0(X,\Omega^\bullet_X) = H^0(X,\Omega_X^2)$, and such that,
		for any closed points $x_1, \ldots, x_m$ and $y_1, \ldots , y_m$  on $X$, 
		$$\sum_{i=1}^{m}[x_i]=\sum_{i=1}^{m} [y_i] \ \mbox{in}\ \CH_0(X)
		 \ \iff \ 
		 \sum_{i=1}^{m}[x_i]=\sum_{i=1}^{m} [y_i] \ \mbox{in}\ \CH_0(X)/F_\varpi^{2m+1}\CH_0(X). $$
	\end{theorem}
	
From there and the work carried out in ~\cite{Vial-JMPA}, we derive explicit criteria as in Theorem~\ref{T2:T} for rational equivalence of effective zero-cycles; we refer to Theorem~\ref{T} for case~\eqref{moduli}, Theorem~\ref{T2} for case~\eqref{kum} and Theorem~\ref{T3} for case~\eqref{fano}.

\subsection{Polynomial decomposition of the diagonal}In the last part of the paper, in an attempt to further explain and motivate Main Conjecture \ref{mainconj}, we introduce in Definition~\ref{polydef}  the notion of polynomial decomposition of the diagonal up to coniveau $c$. A special instance of the definition is the following:
\begin{definition}
A smooth projective $n$-fold $X$ admits a 				degree $l$ polynomial decomposition of the diagonal if
$$\Delta_X=Z_1+Z_2\in \textup{CH}^{n}(X\times X),$$
where $Z_1$ belongs to the subalgebra of $\textup{CH}^{\bullet}(X\times X)$ generated in degree $\leq l$ and $Z_2$ is supported on $D\times X$ for some divisor $D\subset X$.
\end{definition}

In Proposition~\ref{P:poldecalg}, we observe that if $X$ has a degree $l$ polynomial decomposition of the diagonal, then its algebra of holomorphic forms is generated in degrees $\leq l$.
We show in Proposition~\ref{P:polNorimain} that if $X$ has a degree $l$ polynomial decomposition of the diagonal and satisfies Nori's Conjecture \ref{noriconj}, then for $x_1,\ldots, x_m,y_1,\ldots, y_m\in X,$
\begin{equation*}\label{equalitymod}\sum_{i=1}^m x_i \ = \ \sum_{i=1}^m y_i  \ \mbox{in} \ \textup{CH}_0(X) \ \ \iff \ \ \sum_{i=1}^m x_i=\sum_{i=1}^m y_i  \ \mbox{in} \  \textup{CH}_0(X)/F_{\mathrm{V}}^{ml+1}\textup{CH}_0(X).\end{equation*}
Accordingly, the following conjecture in conjunction with Nori's Conjecture \ref{noriconj} implies Main Conjecture~\ref{mainconj}:

\begin{conjecture}[Special instance of Conjecture~\ref{polyconj}]\label{conj:c1}
	Let $l$ be a positive integer.
A smooth projective variety $X$ admits a degree $l$ polynomial decomposition of the diagonal if and only if the algebra $H^{0}(X,\Omega^\bullet)$ is generated in degree $\leq l$.
\end{conjecture}

Finally, we will show in Proposition \ref{impliesgenbloch} how Conjecture \ref{conj:c1} easily implies the generalized Bloch conjecture in coniveau $1$.

\subsection*{Acknowledgments}
We thank Salvatore Floccari for mentioning Remark~\ref{OG6-R=S} and its proof, which led to a proof of Theorem~\ref{HKevidence}\eqref{OG6}. 
We also thank the referee for his detailed comments.
Ch.V.\ is grateful to Stony Brook University for a pleasant stay in March 2023.

\section{Effective zero-cycles on hyper-K\"ahler varieties}
In this section we show that Main Conjecture \ref{mainconj} holds unconditionally for certain hyper-K\"ahler varieties, with respect to a certain Bloch--Beilinson candidate filtration induced by some birational correspondence. 
We also give explicit criteria for the rational equivalence of effective zero-cycles on some hyper-K\"ahler varieties.

\subsection{Hyper-K\"ahler varieties satisfying Main Conjecture \ref{mainconj}}
We start by considering the general situation of an arbitrary smooth projective variety $X$ over an algebraically closed field. We take on the definitions and notation from \cite{Vial-JMPA} concerning birational motives and their co-algebra structure. 
Our main tool used to prove Theorem~\ref{T2:T} and its variants for generalized Kummer varieties and Fano varieties of lines on smooth cubic fourfolds is the observation that \cite[Prop.~5.2]{Vial-JMPA} can be extended to the following:

	\begin{prop}\label{P:cogen}
		Let $X$ be a smooth projective variety over an algebraically closed field and denote $\delta^{k-1} : X \hookrightarrow X^k$ the diagonal embedding, where by convention $\delta^{-1}$ is the structure morphism. 
		Let $\ho_{\varpi}(X) =_{\mathrm{def}} (X,\varpi)$ be a direct summand of the birational motive $\ho(X)$ and assume that there exists $r\geq 0$ such that the morphism $$\bigoplus_{k=0}^r \varpi^{\otimes k} \circ \delta^{k-1}_* : \ \ho(X) \longrightarrow \operatorname{Sym}^{\leq r} \ho_{\varpi}(X)$$ is split injective (we say $\ho(X)$ is \emph{co-generated} by $\ho_{\varpi}(X)$ in degree $\leq r$).
		
		If $x_1, \ldots, x_m$ and $y_1, \ldots , y_m$ are closed points on $X$, then
		$$\sum_{i=1}^{m}[x_i]=\sum_{i=1}^{m} [y_i] \ \mbox{in}\ \CH_0(X) \iff \ \sum_{i=1}^{m} (\varpi_*[x_i])^{\times k} =\sum_{i=1}^{m} (\varpi_*[y_i])^{\times k} \ \mbox{in}\ \CH_0(X^k) \ \mbox{for all } k\leq \min(m,r).$$
	\end{prop}
	\begin{proof}  Under the morphism $\ho(X) \rightarrow \operatorname{Sym}^{\leq r} \ho_{\varpi}(X)$, the class of a closed point $x$ is mapped to $1 + (\varpi)_*[x] +\cdots + (\varpi^{\otimes r})_* \delta^{r-1}_*[x] $. 
		Since $\delta^{k-1}_*[x] = [x]\times \cdots \times [x] \ \mbox{in}\ \CH_0(X^{k})$, we find that
		$$\sum_{i=1}^{m} [x_i] \mapsto m + \sum_{i=1}^{m} \varpi_*[x_i]+ \cdots + \sum_{i=1}^{m} (\varpi_*[x_i])^{\times r}  .$$
		Now, the basic point is that if the morphism $\ho(X) \rightarrow \operatorname{Sym}^{\leq r} \ho_\varpi(X)$ is split injective, then the induced map on Chow groups of zero-cycles is injective. Hence 
		$$\sum_{i=1}^{m}[x_i]=\sum_{i=1}^{m} [y_i] \ \mbox{in}\ \CH_0(X) \ \iff \ \sum_{i=1}^{m} (\varpi_*[x_i])^{\times k} =\sum_{i=1}^{m} (\varpi_*[y_i])^{\times k} \ \mbox{in}\ \CH_0(X^k) \ \mbox{for all } k\leq r.$$
		On the other hand, 
		the fundamental theorem on symmetric polynomials provides the equivalence
		$$\sum_{i=1}^{m} (\varpi_*[x_i])^{\times k} =\sum_{i=1}^{m} (\varpi_*[y_i])^{\times k} \ \forall k\leq r	\iff \ \sum_{i=1}^{m} (\varpi_*[x_i])^{\times k} =\sum_{i=1}^{m} (\varpi_*[y_i])^{\times k}  \ \  \mbox{for all}\ k \leq \min(m,r).$$
		This concludes the proof.
	\end{proof}

If $\ho(X)$ is co-generated by $\ho_{\varpi}(X)$ and if the algebra of global holomorphic form on $X$ is generated in degree $\leq d$, 
recall that for all $j$ and all $1\leq i \leq d$ we consider the descending filtration
	$$F_\varpi^{dj+i} \CH_0(X) =_{\mathrm{def}} \ker \Big( \bigoplus_{k=0}^j \varpi^{\otimes k} \circ \delta^{k-1}_* \ \colon\ \CH_0(X) \longrightarrow \bigoplus_{k=0}^j \CH_0(X^k)\Big),$$
which is finer than the conjectural Bloch--Beilinson filtration.
	Proposition~\ref{P:cogen} admits the following straightforward consequence:
	
	\begin{prop}\label{P:cogen2}
		Let $X$ be a smooth projective variety
		whose algebra of global holomorphic forms is generated in degree $\leq d$. Let $\ho_{\varpi}(X) =_{\mathrm{def}} (X,\varpi)$ be a direct summand of the birational motive $\ho(X)$ and assume that $\ho(X)$ is co-generated by $\ho_{\varpi}(X)$.
		If $x_1, \ldots, x_m$ and $y_1, \ldots , y_m$ are closed points on $X$, then
		$$\sum_{i=1}^{m}[x_i]=\sum_{i=1}^{m} [y_i] \ \mbox{in}\ \CH_0(X) \ \iff \ \sum_{i=1}^{m}[x_i]=\sum_{i=1}^{m} [y_i] \ \mbox{in}\ \CH_0(X)/F_\varpi^{dm+1}\CH_0(X). $$
	\end{prop}

	Recall that the algebra of global holomorphic forms on a hyper-K\"ahler variety is generated in degree~2.
	The following gives evidence for Conjecture~\ref{C} and Main Conjecture \ref{mainconj} in the case of hyper-K\"ahler varieties.\newpage

	\begin{theorem}\label{HKevidence}
		Let $X$ be a hyper-K\"ahler variety. Assume that $X$ is one of the following:
		\begin{enumerate}
			\item $\mathrm{Hilb}^n(S)$, the Hilbert scheme of length-$n$ closed subschemes on a K3 surface $S$\,;
			\item $\mathrm{M}_\sigma(v)$, a moduli space of stable objects on a K3 surface\,;
			\item  $K_n(A)$, the generalized Kummer variety associated to an abelian surface $A$\,;
			\item $F(Y)$, the Fano variety of lines on a smooth cubic fourfold $Y$\,;
			\item  $\widetilde{K}_v(A)$, O'Grady's six-dimensional example~\cite{OG6}.
		\end{enumerate}
		Then there exists a birational idempotent correspondence $\varpi$ such that $\ho(X)$ is co-generated by $\ho_{\varpi}(X)$,   $\varpi_* H^0(X,\Omega^\bullet_X) = H^0(X,\Omega_X^2)$, and such that,
		for any closed points $x_1, \ldots, x_m$ and $y_1, \ldots , y_m$  on $X$, 
		$$\sum_{i=1}^{m}[x_i]=\sum_{i=1}^{m} [y_i] \ \mbox{in}\ \CH_0(X) \ \iff \ \sum_{i=1}^{m}[x_i]=\sum_{i=1}^{m} [y_i] \ \mbox{in}\ \CH_0(X)/F_\varpi^{2m+1}\CH_0(X). $$
	\end{theorem}

		\begin{proof}
		For a hyper-K\"ahler variety $X$, the existence of a birational idempotent correspondence $\varpi$ such that  $\varpi_*H^0(X,\Omega_X^\bullet) = H^0(X,\Omega_X^2)$ and such that
		the morphism $\ho(X) \to \Sym^{\leq n} \ho_{\varpi}(X)$ is an isomorphism (in particular, $\ho(X)$ is co-generated by $\ho_{\varpi}(X)$) is conjectured in \cite[Conj.~2]{Vial-JMPA} (see also \cite[Prop.~5.3]{Vial-JMPA}) and is established in cases \eqref{hilb}, \eqref{moduli}, \eqref{kum}, and \eqref{fano} in \cite[Thm.~5.5]{Vial-JMPA}. 
		Regarding case~\eqref{OG6}, we use below the work of Mongardi--Rapagnetta--Sacc\`a~\cite{MRS} and Floccari~\cite{Floccari} to reduce to case~\eqref{hilb} by showing that the birational motive of $\widetilde{K}_v(A)$ is isomorphic to that of a Hilbert scheme of length-3 closed subschemes on a K3 surface as co-algebra objects. 
		The theorem in all cases listed then follows from Proposition~\ref{P:cogen2}.
		
		O'Grady's hyper-K\"ahler sixfold is obtained as follows. Let $A$ be an abelian surface and let $v=2v_0$ be a Mukai vector on $A$ such that $v_0$ is primitive, effective, and of square 2. Given a $v$-generic polarization $H$, let $\mathrm{M}_v(A)$ be the corresponding moduli space of $H$-semistable sheaves on $A$ and denote $K_v(A)$ its Albanese fiber. 
		Then $K_v(A)$ admits a crepant resolution $\widetilde{K}_v(A)$ which is a hyper-K\"ahler sixfold. 
		By \cite{MRS}, there exists a variety $Y_v(A)$, which by \cite[Prop.~3.3]{Floccari} is birational to a moduli space $\mathrm{M}(A,v)$ of stable sheaves on the Kummer surface associated to $A$,  and a rational map $f: Y_v(A) \dashrightarrow \widetilde{K}_v(A)$ of degree 2. 
		From  the results of \cite{Floccari}, $f_* \colon  \CH_0(Y_v(A)) \to \CH_0(\widetilde{K}_v(A))$ is an isomorphism, and consequently by \cite[Prop.~2.3]{Vial-JMPA} gives an isomorphism $ \ho(\mathrm{M}(A,v)) \cong \ho(\widetilde{K}_v(A))$ of co-algebra objects. 
	\end{proof}
	
	\begin{remark}[The Voisin filtration $S_\bullet$ and the co-radical filtration agree in case~\eqref{OG6}] \label{OG6-R=S} 
		The proof of Theorem~\ref{HKevidence} in 	Case~\eqref{OG6} was made possible after Salvatore Floccari communicated to us the following.
		Let $\widetilde{K}_v(A)$ be O'Grady's six-dimensional example. Together with \cite[Thm.~1$(i)$]{Vial-JMPA}, the arguments in the proof of Theorem~\ref{HKevidence} in that case show that 
		there exists a point $o\in \widetilde{K}_v(A)$ such that, for all $k\geq 0$ and for all $x\in \widetilde{K}_v(A)$, 
		\begin{equation*}
		[x]\in S_k\CH_0(\widetilde{K}_v(A)) \ \iff \ ([x]-[o])^{\times (k+1)} = 0 \ \mbox{in}\ \CH_0(\widetilde{K}_v(A)^{k+1}),
		\end{equation*}
		or equivalently, such that 
		$$S_k\CH_0(\widetilde{K}_v(A)) = R_k\CH_0(\widetilde{K}_v(A))$$
		for all $k\geq 0$. 
		Here, $S_\bullet$ is Voisin's filtration~\cite{V2} and $R_\bullet$ is the co-radical filtration introduced in~\cite[Def.~6.1]{Vial-JMPA} and in~\cite[(33)]{BFMS}.
	\end{remark}

\subsection{Explicit criteria for rational equivalence of effective zero-cycles on hyper-K\"ahler varieties}
In this paragraph we exploit the so-called \emph{strictness} of the explicit co-multiplicative birational Chow--K\"unneth decompositions on the birational motive of certain hyper-K\"ahler varieties, which were constructed in \cite{Vial-JMPA},
 to derive criteria for the coincidence of effective zero-cycles on these hyper-K\"ahler varieties.

\subsubsection{Moduli spaces of stable objects on K3 surfaces}
	Let $S$ be a smooth projective complex K3 surface. For a primitive $v \in H^\bullet(S,\Z)$, and a $v$-generic stability
	condition $\sigma$, let $\mathrm{M}_\sigma(v)$ be the moduli space of $\sigma$-stable complexes on $S$ of Mukai vector~$v$\,; this defines a hyper-K\"ahler variety and we denote by $2n$ its dimension. 
	Marian and Zhao~\cite{MZ} have established the following:
	
	\begin{theorem}[Marian--Zhao \cite{MZ}]\label{MZ}
		Let $\mathcal{F}$ and $\mathcal{G}$ be closed points of $\mathrm{M}_\sigma(v)$. Then
		\begin{align*}
		[\mathcal{F}]= [\mathcal{G}] \ \mbox{in}\ \CH_0(\mathrm{M}_\sigma(v)) 
		\iff   c_2(\mathcal{F}) = c_2(\mathcal{G})\ \mbox{in}\ \CH_0(S).
		\end{align*}
	\end{theorem}
	
	We have the following generalization to effective zero-cycles of arbitrary degree:

	\begin{theorem} \label{T}
		Let $\mathcal{F}_1, \ldots, \mathcal{F}_m$ and $\mathcal{G}_1,\ldots, \mathcal{G}_m$ be closed points of $\mathrm{M}_\sigma(v)$. Then
		\begin{align*}
		&	\sum_{i=1}^{m}[\mathcal{F}_i]=\sum_{i=1}^{m} [\mathcal{G}_i] \ \mbox{in}\ \CH_0(\mathrm{M}_\sigma(v)) \\
		& \iff  \sum_{i=1}^{m} c_2(\mathcal{F}_i)^{\times k} =\sum_{i=1}^{m} c_2(\mathcal{G}_i)^{\times k} \ \mbox{in}\ \CH_0(S^k) \ \mbox{for all } k\leq \min (m,n).
		\end{align*}
	\end{theorem}

	\begin{proof}
		Denote by $[o_S]$ the Beauville--Voisin class on $S$ and let $c\in \mathbb Z$ be the constant (which depends only on the Mukai vector $v$) such that $\deg c_2(\mathcal F) = c$ for all $\mathcal F \in \mathrm{M}_\sigma(v)$. We denote $\ho_2(S)$ the image of the birational idempotent correspondence $(\Delta_S -  S\times [o_S])|_{\eta_S \times S}$ acting on $\ho(S)$.
		The birational motive $\ho(\mathrm{M}_\sigma(v))$ is then canonically isomorphic to $\operatorname{Sym}^{\leq n}\ho_\varpi(\mathrm{M}_\sigma(v))$ for a birational idempotent correspondence $\varpi$ factorizing as $\varpi : \ho(\mathrm{M}_\sigma(v)) \twoheadrightarrow \ho_2(S) \hookrightarrow \ho(\mathrm{M}_\sigma(v))$ with the left arrow satisfying $[\mathcal F] \mapsto c_2(\mathcal F) - c[o_S]$ for all $\mathcal F \in \mathrm{M}_\sigma(v)$. See the proofs of \cite[Thm.~3.1]{Vial-JMPA} and \cite[Thm.~5.5]{Vial-JMPA}, which are based on the theorem of Marian--Zhao.
		From Proposition~\ref{P:cogen} it follows that
		\begin{align*}
		&	\sum_{i=1}^{m}[\mathcal{F}_i]=\sum_{i=1}^{m} [\mathcal{G}_i] \ \mbox{in}\ \CH_0(\mathrm{M}_\sigma(v)) \\
		& \iff  \sum_{i=1}^{m} (c_2(\mathcal{F}_i)-c[o_S])^{\times k} =\sum_{i=1}^{m} (c_2(\mathcal{G}_i)-c[o_S])^{\times k} \ \mbox{in}\ \CH_0(S^k) \ \mbox{for all } k\leq \min(m,n).
		\end{align*}
		The latter is easily seen to be further equivalent to $\sum_{i=1}^{m} c_2(\mathcal{F}_i)^{\times k} =\sum_{i=1}^{m} c_2(\mathcal{G}_i)^{\times k} \ \mbox{in}\ \CH_0(S^k) $ for all $k\leq \min (m,n).$
	\end{proof}

	\subsubsection{Generalized Kummer varieties} Let $A$ be an abelian surface. Recall that the $2n$-dimensional generalized Kummer variety $K_n(A)$ is the fiber over $0$ of the composition of the Hilbert--Chow morphism with the sum map $\mathrm{Hilb}^{n+1}(A) \to A^{n+1}/\mathfrak{S}_{n+1} \to A$.
We thus have a natural morphism $K_n(A) \to A^{n+1}_0/\mathfrak{S}_{n+1}$, where $A^{n+1}_0 =_{\mathrm{def}} \ker (\Sigma : A^{n+1} \to A)$.
Let us denote by $\{x_1,\ldots,x_{n+1}\}$ the closed points of $A^{n+1}_0/\mathfrak{S}_{n+1}$; these are unordered $(n+1)$-tuple of closed point of $A$ such that $x_1 + \cdots + x_{n+1}=0$ in $A$. By \cite[\S 6]{ftv}, the pushforward map $\CH_0(K_n(A)) \to \CH_0( A^{n+1}_0/\mathfrak{S}_{n+1})$ is an isomorphism. Therefore the rational equivalence class of a point in $K_n(A)$ only depends on its image in $ A^{n+1}_0/\mathfrak{S}_{n+1}$ and we have a canonical isomorphism $\ho(K_n(A)) \cong \ho(A^{n+1}_0/\mathfrak{S}_{n+1})$.

	\begin{theorem} \label{T2}
		Let $p_1,\ldots,p_m$ and $q_1,\ldots, q_m$ be closed points of $K_n(A)$ with respective images \linebreak 
		$\{x_{1,1},\ldots,x_{1,n+1}\},\ldots, \{x_{m,1},\ldots,x_{m,n+1}\}$ and $\{y_{1,1},\ldots,y_{1,n+1}\},\ldots,\{y_{m,1},\ldots,y_{m,n+1}\}$  in $ A^{n+1}_0/\mathfrak{S}_{n+1}$. 
		Then
		\begin{align*}
		&	\sum_{i=1}^{m}[p_i]=\sum_{i=1}^{m} [q_i] \ \mbox{in}\ \CH_0(K_n(A)) \\
		& \iff  \sum_{i=1}^{m} \Big(\sum_{j=1}^{n+1} [x_{i,j}]\Big)^{\times k} =\sum_{i=1}^{m} \Big(\sum_{j=1}^{n+1} [y_{i,j}]\Big)^{\times k} \ \mbox{in}\ \CH_0(A^k) \ \mbox{for all } k\leq \min (m,n).
		\end{align*}
	\end{theorem}
	\begin{proof}
We equip the Chow motive $\h(A)$ of $A$ with its Deninger--Murre Chow--K\"unneth decomposition (which we consider covariantly)
$$\h(A) = \bigoplus_{i=0}^4 \h_i(A), \quad \mbox{with } \h_i(A) =_{\mathrm{def}} (A,\varpi^A_i).$$
 By construction, this decomposition is an eigenspace decomposition for the multiplication by $r$ maps $[r]: A \to A, a \mapsto ra$.
 The action on zero-cycles satisfies $$(\varpi^A_0)_*[a] = [0], \quad (\varpi^A_1)_*[a] = \frac{1}{2}([a]-[-a])\quad  \mbox{and}\quad  (\varpi^A_2)_*[a] = \frac{1}{2}([a]+[-a])-[0]$$ for all closed points $a \in A$. 
The product Chow--K\"unneth decomposition on $A^l$ (which coincides with the Deninger--Murre Chow--K\"unneth decomposition on $A^l$) 
then also provides an eigenspace decomposition for the multiplication by $r$ maps on $A^l$ and 
hence provide a Chow--K\"unneth decomposition for $\h(A^l/\mathfrak{S}_l)$. 
Likewise, identifying $A^{l+1}_0$ with $A^l$ and noting that the sum map commutes with the multiplication by $r$ maps on $A^l$, 
the product Chow--K\"unneth decomposition on $A^l$ provides a Chow--K\"unneth decomposition for $\h(A_0^{l+1}/\mathfrak{S}_{l+1})$.
Via the identification $\ho(K_n(A)) \cong \ho(A_0^{n+1}/\mathfrak{S}_{n+1})$, 
this provides a birational Chow--K\"unneth decomposition $\ho(K_n(A)) = \bigoplus_{i=0}^n \ho_{2i}(K_n(A))$ and it is shown in \cite[Thm.~5.5(iii)]{Vial-JMPA} that the canonical morphism
$$\ho(K_n(A)) \stackrel{\simeq}{\longrightarrow} \mathrm{Sym}^{\leq n}\ho_2(K_n(A))$$ is a graded isomorphism.
On the other hand, the natural embedding $\iota  : A_0^{n+1}/\mathfrak{S}_{n+1} \hookrightarrow A^{n+1}/\mathfrak{S}_{n+1}$ commutes with the multiplication by $r$ maps and thus induces a graded morphism of Chow motives 
$$\iota_*  \colon \h(A_0^{n+1}/\mathfrak{S}_{n+1}) \longrightarrow \h(A^{n+1}/\mathfrak{S}_{n+1}).$$
Moreover, its restriction to the generic point 
$$\iota_* \colon \ho(K_n(A)) \simeq \ho(A_0^{n+1}/\mathfrak{S}_{n+1}) \longrightarrow \ho(A^{n+1}/\mathfrak{S}_{n+1})$$ is split injective. 
Indeed, 
if $$\Gamma =_{\mathrm{def}} A^{(n,n+1)} =_{\mathrm{def}} \{(\{x_1,\ldots, x_n\}, \{x_1,\ldots,x_n,x_{n+1}\}) \  \vert \ x_1,\ldots, x_{n+1}\in A\} \subset A^n/\mathfrak{S}_n \times A^{n+1}/\mathfrak{S}_{n+1}$$
denotes the incidence correspondence, we have 
$\frac{1}{n+1}\Gamma^*\iota_*([p]) = [p]$ in $\CH_0(K_n(A))$ for all closed points $p\in K_n(A)$.

Therefore, by Proposition~\ref{P:cogen}, 	$	\sum_{i=1}^{m}[p_i]=\sum_{i=1}^{m} [q_i] \ \mbox{in}\ \CH_0(K_n(A))$ if and only if 
$$ \sum_{i=1}^{m} \big( (\varpi^{A^{n+1}}_2)_*\iota_*[p_i] \big)^{\times k} 
=\sum_{i=1}^{m} \big( (\varpi^{A^{n+1}}_2)_*\iota_*[q_i] \big)^{\times k} \
 \mbox{in}\ \CH_0((A^{n+1})^k) \ \mbox{for all } k\leq \min (m,n),$$
 where 
 \begin{align*}
 \varpi^{A^{n+1}}_2 = & \underbrace{\big(\varpi^A_2 \otimes \varpi^A_0 \otimes \cdots \otimes \varpi^A_0 +  \mbox{(sym)}\big)}_{\varpi_{2,0}} + \underbrace{\big(\varpi^A_1 \otimes \varpi^A_1 \otimes \varpi^A_0 \otimes \cdots \otimes \varpi^A_0 +  \mbox{(sym)}\big)}_{\varpi_{1,1}}.
 \end{align*}

Since $(\varpi_{2,0})_*\iota_*[p_i] = c\sum_{j=1}^{n+1} [\{x_{ij},0,\ldots,0\}]$ for some non-zero combinatorial constant $c$ only depending on $n$ (and similarly for $q_i$ in place of $p_i$), we only need to show that, for any point $x= \{x_1,\ldots, x_{n+1}\}$ in $A_0^{n+1}/\mathfrak{S}_{n+1}$,
$(\varpi^{A^{n+1}}_2)_*[x]$ vanishes if and only if $(\varpi_{2,0})_*[x]$ vanishes.
The ``only if'' part is clear since the idempotents $\varpi_{2,0}$ and $\varpi_{1,1}$ are orthogonal. For the ``if'' part, we first note that
$(\varpi_{2,0})_* [x] $ is equal to the symmetrization of $d(\sum_{i=1}^{n+1} [x_i]-(n+1)[0])\times [0]^{\times n}$ in $\CH_0(A^{n+1})$ for some non-zero combinatorial constant $d$. In particular, if 
$(\varpi_{2,0})_*[x]$ vanishes,
then $\sum [x_i] = (n+1)[0]$ in $\CH_0(A)$.
Second, the identity $\sum_i x_i = 0$ in $A$ implies \cite{lin-corr} that $\sum_i [x_i] = \sum_i[-x_i]$ in $\CH_0(A)$, i.e., that $(\varpi^A_1)_*\sum_i [x_i] = 0$ in $\CH_0(A)$. Now
$(\varpi_{1,1})_*[x]$ is a non-zero multiple of the symmetrization of
 $$\Big( \sum_i (\varpi^A_1)_*[x_i] \times \sum_i (\varpi^A_1)_*[x_i] - \sum_i (\varpi^A_1)_*[x_i] \times (\varpi^A_1)_*[x_i] \Big)\times [0]^{\times (n-1)}  $$ 
  in $\CH_0(A^{n+1})$.
Thus if $x_1,\ldots, x_{n+1}$ are closed points in $A$ such that $x_1 + \cdots + x_{n+1} = 0$ in $A$, then 
$(\varpi_{1,1})_*[x]$ is a non-zero multiple of the symmetrization of 
$$\sum_i (\varpi^A_1)_*[x_i] \times (\varpi^A_1)_*[x_i]\times [0]^{\times (n-1)} = \Big((\varpi^A_1 \times \varpi^A_1)_* \delta_* \sum_i[x_i] \Big) \times [0]^{\times (n-1)},$$
where $\delta: A \hookrightarrow A\times A$ is the diagonal embedding. It  then clearly follows that if $\sum_i [x_i] = (n+1)[0]$, then $(\varpi_{1,1})_*[x]$ vanishes and hence that $(\varpi_2^{A^{n+1}})_*[x]$ vanishes.
	\end{proof}

		\subsubsection{Fano varieties of lines on cubic fourfolds} Let $Y$ be a smooth cubic hypersurface in $\mathbb{P}^5_\mathbb{C}$ and let $X=F(Y)$ be the Fano variety of lines on $Y$; it is a hyper-K\"ahler variety of dimension~4 polarized by the restriction $g$ of the Pl\"ucker polarization on the Grassmannian $\textup{Gr}(\PP^1,\PP^5)$. By abuse we will denote both by $l$ a line in $Y$ and the corresponding point in $X$.
	
	\begin{theorem} \label{T3}
			Let $X = F(Y)$ be the Fano variety of lines on a smooth cubic fourfold $Y$. Then
	$$[l]=[l'] \ \mbox{in}\ \CH_0(F(Y)) \ \iff \ [l]=[l'] \ \mbox{in}\ \CH_1(Y)$$
	and, for $m>1$,
		$$\sum_{i=1}^m[l_i]=\sum_{i=1}^m[l'_i] \ \mbox{in}\ \CH_0(F(Y)) 
		\ \iff \
		\begin{cases}
	\sum_{i=1}^m[l_i]=\sum_{i=1}^m[l'_i]  \ \mbox{in}\ \CH_1(Y), \mbox{ and}\\
		\sum_{i=1}^m[l_i]\times [l_i]=\sum_{i=1}^m[l'_i] \times [l'_i] \ \mbox{in}\ \CH_2(Y\times Y).
		\end{cases}$$
	\end{theorem}
	\begin{proof} Recall from \cite[Thm.~21.9]{SV} that the action of Voisin's birational map $\varphi : X \dashrightarrow X$ on $\CH_0(X)$ diagonalizes.  
	We then consider the birational idempotent correspondence  $\varpi$ on $\ho(X)$ given by the projector with respect to the eigenspace decomposition of $\varphi$ on the eigenspace corresponding to the eigenvalue $-2$. By \cite[Thm.~5.5(iv)]{Vial-JMPA}, the canonical morphism $\ho(X) \to \operatorname{Sym}^{\leq 2} \ho_\varpi(X)$ is an isomorphism.  
		
   Let now $Z =_{\mathrm{def}} \{(y,l) \in Y\times X \ \vert \ y\in l \}$ be the universal line.
   By the Franchetta conjecture for $X\times X$ \cite{FLV-Franchetta}, $\varpi$ is in fact the restriction to the generic point of the idempotent correspondence 
   $$\begin{tikzcd}\h^6(X)_{\mathrm{prim}}  \ar[r,"Z_*"] & \h^4(Y)_{\mathrm{prim}}(-1)\ar[rr,"-\frac{1}{6} (p_2^*g^2) \circ Z^*"] & &\h^6(X)_{\mathrm{prim}}.\end{tikzcd}$$
   Moreover both arrows are isomorphisms of Chow motives. Here $\h^4(Y)_{\mathrm{prim}}$ is the direct summand of the Chow motive $\h(Y)$ cut out by the idempotent $p=_{\mathrm{def}}\Delta_Y - \frac{1}{3}\sum_{i=0}^4 h^i\times h^{4-i}$ with $h \in \CH^1(Y)$ the hyperplane class, and $\h^6(X)_{\mathrm{prim}}$ is the direct summand of $\h(X)$ cut out by the idempotent correspondence $Z_*\circ p\circ \big(-\frac{1}{6} (p_2^*g^2) \circ Z^*\big)$.
	We can then conclude by Proposition~\ref{P:cogen} after noting that
	$p\circ Z_* [l] = [l]-[l_0]$ where $l_0$ is any line on $Y$ with class $\frac{1}{3}h^3$.
	\end{proof}

\section{The Voisin filtration}
In this section we show that Main Conjecture \ref{mainconj} holds unconditionally for the Beauville candidate for the Bloch--Beilinson filtration for abelian varieties, and  present some cases in which Main Conjecture \ref{mainconj} is satisfied for the Voisin candidate for the Bloch--Beilinson filtration, assuming well-known conjectures on algebraic cycles.

\subsection{Main Conjecture \ref{mainconj} for abelian varieties}
In \cite{V3}, Voisin shows that given an abelian variety $A$, a desingularization of the quotient $A/\pm$ satisfies Voisin's Conjecture \ref{Vconj}. This is essentially equivalent to the fact that abelian varieties satisfy Main Conjecture \ref{mainconj} for $m=2$, and the proof can be adapted to show that the following conjecture of Nori implies Main Conjecture \ref{mainconj} for abelian varieties.

\begin{noriconjecture}[\cite{N}]\label{noriconj}
Let $X$ be a smooth projective variety and $w\in \textup{CH}^i(X)$ a cycle such that $w|_T=0\in \textup{CH}_0(T)$ for any $i$-fold $T\subset X$. Then $w=0\in \textup{CH}^i(X)$.
\end{noriconjecture}

\begin{proposition}\label{abprop}
Main Conjecture \ref{mainconj} holds for abelian varieties and $m=2$. Moreover, Nori's Conjecture \ref{noriconj} for an abelian $g$-fold $A$ implies Main Conjecture \ref{mainconj} for $A$.
\end{proposition}
\begin{proof}
Suppose \[[a]+[a']=[b]+[b'] \ \mbox{in}\ \textup{CH}_0(A)/F^3_{\mathrm{V}}\textup{CH}_0(A).\]
Then $a+a'=b+b'$ so, translating by $-a-a'$ if needed, we can assume that $a'=-a$, $b'=-b$. Voisin shows in \cite[Prop.~2.17]{V3} that
$$[a]+[-a]=[b]+[-b]\ \mbox{in}\ \textup{CH}_0(A)/F^3_{\mathrm{BB}}\textup{CH}_0(A)\ \iff \  [a]+[-a]=[b]+[-b]\ \mbox{in}\ \textup{CH}_0(A).$$

For the general case we can follow a similar argument. The main differences are the use of the fundamental theorem on symmetric polynomials and the substitution of Nori's Conjecture \ref{noriconj} in place of a theorem of Joshi \cite{J}. Suppose that 
\begin{equation}\label{equal}
\sum_{i=1}^m [x_i]=\sum_{i=1}^m [y_i]\ \mbox{in}\ \textup{CH}_0(A)/F_{\mathrm{BB}}^{m+1}\textup{CH}_0(A).
\end{equation}
Let $\Theta$ be an ample divisor that gives an isogeny
\begin{align*}
A& \ \longrightarrow \ \ \ \ \ \  \ \ \ \   \widehat{A}\\
x& \ \longmapsto  \ D_x=_{\textup{def}}\Theta_x-\Theta.
\end{align*}
The map $x\mapsto D_x^j\in \textup{CH}^j(A)$ is a given by a correspondence in $\textup{CH}^j(A\times A)$. Hence, assuming Nori's Conjecture \ref{noriconj}, we can use Lemma \ref{corresp} and \eqref{equal} to deduce that
$$\sum_{i=1}^mD_{x_i}^j=\sum_{i=1}^m D_{y_i}^j,\qquad \forall j\in \{1,\ldots, m\}.$$
The fundamental theorem on symmetric polynomials then implies that
$$\sum_{i=1}^mD_{x_i}^j=\sum_{i=1}^m D_{y_i}^j \ \mbox{in}\ \textup{CH}^j(A),\qquad \forall j\in \mathbb{N}.$$
The Chow ring of $A$ has two ring structures, one given by intersection and the other by the Pontryagin product:
\begin{align*}\textup{CH}^\bullet(A)\times \textup{CH}^\bullet(A) \ \longrightarrow \ &  \ \textup{CH}^\bullet(A)\\
(\alpha,\beta) \ \ \ \ \ \ \ \ \ \ \longmapsto \  & \Sigma_*(\alpha\times \beta),
\end{align*}
where $\Sigma: A\times A\longrightarrow A$ is the sum map. 
A formula of Beauville then gives
$$\sum_{i=1}^m\frac{\Theta^{g-j}}{(g-j)!} D_{x_i}^j \ = \ \sum_{i=1}^m \frac{\Theta^g}{g!}*\gamma(x_i)^{*j}  \ \mbox{in} \  \textup{CH}_0(A),\qquad \forall j\in \mathbb{N},$$
where 
$$\gamma(x)=\sum_{k=1}^n\frac{1}{k}([x]-[0_A])^{*k}$$
is the logarithm of $[x]$. Since $\text{exp}(\gamma(x))=[x]$ and we can assume that $\Theta^g/g!=d[0_A]$ for some positive integer $d$, we get
$$\sum_{i=1}^m [x_i]=\sum_{i=1}^m\text{exp}(\gamma(x_i))=\sum_{i=1}^m\text{exp}(\gamma(y_i))=\sum_{i=1}^m [y_i]\in \textup{CH}_0(A).$$
\end{proof}

Instead of considering the filtration $F_{\mathrm{V}}^{\bullet}$, we can consider the Beauville filtration $F_{\mathrm{B}}^{\bullet}$ on the Chow ring of an abelian variety (see \cite{Beau2}) to obtain an unconditional proof of an analogue of Main Conjecture~\ref{mainconj}:

\begin{theorem}\label{uncondabvar}
Let $A$ be an abelian $g$-fold and let $x_1,\ldots, x_m,y_1,\ldots, y_m\in A$,
\begin{align*}
 \sum_{i=1}^m [x_i]=\sum_{i=1}^m [y_i]   \ \mbox{in} \ \textup{CH}_0(A)\iff \sum_{i=1}^m [x_i]-\sum_{i=1}^m [y_i]  \ \mbox{lies in} \ \bigoplus_{s=m+1}^g\textup{CH}^g_{(s)}(A).
\end{align*}
\end{theorem}
\begin{proof}
The map $x\mapsto D_x^j\in \textup{CH}^j(A)$ as above is a given by a correspondence $\Gamma_j\in \textup{CH}^j(A\times A)$. Since $\textup{CH}_{(s)}^i(A)=0$ for all $s>i$,
$${\Gamma_j}_* (\textup{CH}_{(s)}^g(A))=0,\qquad \forall s\geq j.$$
Thus, if $\sum_{i=1}^m [x_i]-\sum_{i=1}^m [y_i] \in \bigoplus_{s=m+1}^g\textup{CH}^g_{(s)}(A)$, 
$$\sum_{i=1}^mD_{x_i}^j=\sum_{i=1}^m D_{y_i}^j,\qquad \forall j\in \{1,\ldots, m\}.$$
We can then apply the same reasoning as in the proof of Proposition \ref{abprop}.
\end{proof}

\begin{remark}
	Let $\h(A)=\bigoplus_{i=0}^{2g} \h_i(A)$, with $\h_i(A) = (A,\varpi_i)$, be the Deninger--Murre decomposition of the Chow motive of $A$.
	If one takes $\varpi$ to be $\varpi_1|_{\eta_A \times A}$, where $\eta_A$ is the generic point of $A$, then by K\"unnemann~\cite{K-abelian} $\ho(A)$ is co-generated by $\ho_\varpi(A)$ and the filtration $F^\bullet_\varpi$ is the Beauville filtration, so that Proposition~\ref{P:cogen2} recovers Theorem~\ref{uncondabvar}.
\end{remark}

\subsection{Deducing Main Conjecture \ref{mainconj} for hyper-K\"ahler varieties from well-known conjectures on algebraic cycles}

In \cite{V3}, the author shows how the nilpotence conjecture implies Voisin's Conjecture \ref{Vconj} for hyper-K\"ahler varieties satisfying the Lefschetz standard conjecture in degree $2$. In \cite[Rem.~2.14]{V3} she specifies that her proof does not imply a stronger version of the statement of Main Conjecture \ref{mainconj}, where the depth in the Bloch--Beilinson does not depend on the degree $m$ of the effective zero-cycles. 
In short the reason is that $(s-t)\in\mathbb{Z}[s,t]$ divides $(s^n-t^n)$ for all $n>0$, whereas 
$$\sum_{i=1}^m s_i-\sum_{i=1}^m t_i  \ \mbox{in} \  \mathbb{Q}[s_1,\ldots, s_m,t_1,\ldots, t_m]$$
need not divide
$$\sum_{i=1}^m s_i^n-\sum_{i=1}^m t_i^n.$$
Nonetheless, one can adapt Voisin's argument using the fundamental theorem on symmetric polynomials and Nori's Conjecture \ref{noriconj}. The main consequence of Nori's Conjecture \ref{noriconj} we will need is the following:
\begin{lemma}[{\cite[Lem.~2.8]{V3}}]\label{corresp}
If Nori's Conjecture \ref{noriconj} holds, then for any smooth projective varieties $X,Y$, correspondence $\Gamma\in \textup{CH}^i(X\times Y)$, and $w\in F^{i+1}_{\mathrm{BB}}\textup{CH}_0(X)$, we have $\Gamma_* w=0$.
\end{lemma}

\begin{proposition}\label{HKprop}
Let $X$ be a hyper-K\"ahler variety which satisfies Nori's Conjecture \ref{noriconj} and the Lefschetz standard conjecture in degree $2$. Then the nilpotence conjecture implies Main Conjecture~\ref{mainconj} for $X$.
\end{proposition}

\begin{proof}
This proof is a simple adaptation of Voisin's argument in Section 2.2 of \cite{V3} and we use the same notation. $X$ is a smooth projective hyper-K\"ahler variety of dimension $2n$ and we fix a polarizing class $h_X\in \text{Pic}(X)$. We let $Z_{\text{lef}}\in \textup{CH}^2(X\times X)$ be a cycle whose existence is predicted by the Lefschetz standard conjecture in degree $2$, namely such that $[Z_{\text{lef}}]^*: H^{4n-2}(X)\longrightarrow H^2(X)$ is the inverse of the cup product map with $h_X^{2n-2}$. As shown in \cite{V3}, there is a polynomial in $Z_{\text{lef}}$ and $\text{pr}_2^*h_X$ 
$$P \ = \ \sum_{i=0}^n \mu_iZ_{\text{lef}}^i\cdot \text{pr}_2^* h_X^{2n-2i}  \ \mbox{in} \  \textup{CH}^{2n}(X\times X)$$
such that $[P]^*$ acts as the identity on holomorphic forms. Reasoning as in \cite{V3} and assuming the nilpotence conjecture, one sees that the map $P_*: \textup{CH}_0(X)\longrightarrow \textup{CH}_0(X)$ is the identity.
\medskip

Suppose that $x_1,\ldots, x_m,y_1,\ldots, y_m\in X$ and
\[\sum_{i=1}^m [x_i]-\sum_{i=1}^m [y_i]\in F^{2m+1}_{\mathrm{V}}\textup{CH}_0(X).\]
Then by Nori's Conjecture \ref{noriconj}
$$\sum_{i=1}^m(Z_{\text{lef}}^j)_{x_i} \ = \ \sum_{i=1}^m(Z^j_{\text{lef}})_{y_i}  \ \mbox{in} \  \textup{CH}^{2i}(X),\qquad j=1,\ldots, m.$$
Here, for a correspondence $Z \in \CH^i(X\times X)$ and a closed point $x\in X$, we write $Z_x$ for $Z_*[x]$.
Denoting by $\iota_x: X\longrightarrow X\times X$ the embedding given by $\iota_x(y)=(x,y)$, we have
\begin{align*}Z_{\text{lef}}^j\cdot \{x\}\times X&=Z_{\text{lef}}^j\cdot {\iota_x}_*(X)= {\iota_x}_*(\iota_x^*(Z_{\text{lef}}^j))= {\iota_x}_*(\iota_x^*(Z_{\text{lef}})^j)\\&={\iota_x}_*\big((Z_{\text{lef},x})^j\big)=\text{pr}_1^*(\{x\})\cdot\text{pr}_2^*\big((Z_{\text{lef},x})^j\big).\end{align*}
Accordingly,
$$(Z_{\text{lef}}^j)_x=(Z_{\text{lef},x})^j,\qquad  j=0,\ldots, n.$$
This gives
$$\sum_{i=1}^m(Z_{\text{lef}, x_i})^j=\sum_{i=1}^m(Z_{\text{lef}, y_i})^j,\qquad j=0,\ldots, m.$$
By the fundamental theorem on symmetric polynomials the same equality holds for all $j$ and thus
$$\sum_{i=1}^m(Z_{\text{lef}}^j)_{x_i}=\sum_{i=1}^m(Z_{\text{lef}}^j)_{y_i},\qquad j=0,\ldots, n.$$
Finally, we conclude that
$$\sum_{i=1}^m [x_i]=P_*\left(\sum_{i=1}^m[x_i]\right)=P_*\left(\sum_{i=1}^m[y_i]\right)=\sum_{i=1}^m[y_i]  \ \mbox{in} \  \textup{CH}_0(X).$$
\end{proof}

\section{Polynomial decompositions of the diagonal }

\subsection{Definition and examples}

In \cite{V3}, Voisin formulates the following conjecture, which implies Voisin's Conjecture~\ref{Vconj} for hyper-K\"ahler varieties with respect to the Voisin filtration; see \cite[Prop.~2.15]{V3}.

\begin{conjecture}[{\cite[Conj.~2.16]{V3}}]\label{HKdiag}Consider a smooth projective hyper-K\"ahler $2n$-fold~$X$, 
a polarization $h_X$, and a cycle $Z_{\textup{lef}}\in \textup{CH}^2(X\times X)$ such that $[Z_{\textup{lef}}]^*$ is the inverse of the cup product with $h_X^{2n-2}$. There are cycles $\gamma_i\in \textup{CH}^{2n-2i}(X)$, $i=0,\ldots, n$, a divisor $D\subset X$, and a cycle $W\in \textup{CH}^{2n}(X\times X)$ supported on $D\times X$ such that
$$\Delta_X = \sum_{i=0}^n Z_{\textup{lef}}^i\cdot \textup{pr}_2^*(\gamma_i)+W \in  \textup{CH}^{2n}(X\times X).$$
\end{conjecture}

This conjecture suggests the following definition:

\begin{definition}\label{polydef} 
A smooth projective variety $X$ admits a \emph{degree $l$ polynomial decomposition of the diagonal up to coniveau $c$} if
$$\Delta_X=Z_1+Z_2\in \textup{CH}^{n}(X\times X),$$
where $Z_1$ belongs to the subalgebra of $\textup{CH}^{\bullet}(X\times X)$ generated in degree at most $l$ and $Z_2$ is supported on $Y\times X$ for some closed $Y\subset X$ of codimension $c$. If $c$ can be taken to be positive, we say that $X$ has a \emph{degree $l$ polynomial decomposition of the diagonal}.
\end{definition}

\begin{examples}\label{ex1} \hfill
	
\begin{enumerate}[(i)]
\item $\mathbb{P}^n$ has a degree $1$ polynomial decomposition of the diagonal. Indeed, writing $h$ for $c_1(\mathcal{O}(1))$, we have
$$\Delta_{\mathbb{P}^n}=\sum_{i=0}^n\textup{pr}_1^*(h^i)\cdot \textup{pr}_2^*(h^{n-i})\in \CH^n(\mathbb{P}^n\times \mathbb{P}^n).$$

\item A variety with a rational decomposition of the diagonal in the sense of Bloch--Srinivas \cite{BlochSrinivas} has a polynomial decomposition of the diagonal in degree $1$ up to coniveau $1$.
\item Curves have a degree $1$ polynomial decomposition of the diagonal. More generally, any $n$-fold has a degree $n$ polynomial decomposition of the diagonal.

\item A surface $X$ with $p_g(X)=0$ has a degree $1$ polynomial decomposition of the diagonal if and only if it satisfies Bloch's conjecture. Murre \cite{Murre} defines a decomposition of the motive $h(X)$
$$h(X)=\sum_{i=0}^4 h^i(X).$$
The motives $h^i(X),\ i\neq 2$ are cut out by idempotent correspondences which are products of divisors. The motive $h^2(X)$ further breaks up as a sum $h^2_{\textup{alg}}(X)+h^2_{\textup{tr}}(X)$ \cite{KMP}. If Bloch's conjecture holds for $X$, then $h^2_{\textup{tr}}(X)=0$ and thus $\Delta$ is a polynomial in divisor classes. We will explain the converse in Proposition \ref{impliesgenbloch}.

\item An abelian variety $A$ has a degree $1$ polynomial decomposition of the diagonal. Indeed, the diagonal of $A$ is a rational multiple of $f^{*}(\Theta)^{\dim A}$, where $f: A^2\longrightarrow A$ is given by $f(a,b)=a-b$, and $\Theta$ is a symmetric ample divisor.
\item If $X_1,\ldots, X_r$ have polynomial decompositions of degree $l_1,\ldots, l_r$ up to coniveau $c_1,\ldots, c_r$ then $X_1\times\cdots\times X_r$ has a degree $\max_{1\leq i\leq r}(l_i)$ polynomial decomposition of the diagonal up to coniveau $\textup{max}_{1\leq i\leq r}(c_i)$.
\end{enumerate}
\end{examples}

\begin{example}\label{q=0}
Let $X$ be a smooth projective variety with $H^0(X,\Omega^1)=0$. Since 
\[\textup{Pic}(X\times X)=\textup{pr}_1^*\textup{Pic}(X)\oplus \textup{pr}_2^*\textup{Pic}(X),\]
$X$ has a degree $1$ polynomial decomposition decomposition of the diagonal up to coniveau $1$ if and only if $X$ has a rational decomposition of the diagonal. By Proposition 1 of \cite{BlochSrinivas} this is the case if and only if the degree map $\CH_0(X)\longrightarrow \mathbb{Q}$ is an isomorphism.
\end{example}

The generalized Bloch conjecture predicts when the degree map $\CH_0(X)\longrightarrow \mathbb{Q}$ is an isomorphism.

\begin{conjecture}[generalized Bloch conjecture, see e.g.\ \cite{VoisinHodge}]\label{genBloch}
Let $X$ be a smooth projective $n$-fold such that $H^{p,q}(X)=0$ for all $p\neq q$ and $p<c$. Then the cycle class map
$$cl: \CH_i(X)\longrightarrow H^{2n-2i}(X,\mathbb{Q})$$
is injective for all $i<c$.
\end{conjecture}

\subsection{Conjecture and relationship with the generalized Bloch conjecture}

Example \ref{q=0} explains how the generalized Bloch conjecture predicts that a variety with $H^0(X,\Omega^1)=0$ has a degree $1$ polynomial decomposition of the diagonal up to coniveau $1$ if and only if
$$H^k(X,\mathbb{Q})=N_H^1H^k(X,\mathbb{Q}) \qquad \forall k>0,$$
where $N_H^\bullet$ denotes the Hodge coniveau filtration. Recall that $N_H^rH^{\bullet}(X,\mathbb{Q})$ is by definition the largest Hodge substructure $V\subset H^\bullet(X,\mathbb{Q})$ which has coniveau at least $r$, namely such that $V_{\mathbb{C}}^{p,q}=0$ if $p< r$.
\medskip

This provides evidence that a polynomial decomposition of the diagonal can be detected by Hodge theory. 
\begin{conjecture}\label{polyconj}
Let $X$ be a smooth projective $n$-fold. Then $X$ admits a degree $l$ polynomial decomposition of the diagonal up to coniveau $c$ if and only if
$$N_H^rH^\bullet(X,\mathbb{Q})/N_H^{r+1}H^\bullet(X,\mathbb{Q})$$
is generated in degree at most $ l+r$  for all $r< c$.
\end{conjecture}
\begin{remark}
It suffices to check this in degree $\leq n$ as the hard Lefschetz theorem then implies the same generation statement in high degree. Example \ref{ex1} (4) illustrates why the shift in degree of generation according to coniveau is necessary. For example, a smooth cubic surface has a degree $1$ polynomial decomposition of the diagonal but its primitive cohomology is not generated in degree $1$.
\end{remark}

As is often the case, it is easy to deduce Hodge-theoretic information from cycle-theoretic information:
\begin{proposition}\label{P:poldecalg}
 If $X$ has a degree $l$ decomposition of the diagonal up to coniveau $c$ the algebra
$$N_H^rH^{\bullet}(X,\mathbb{Q})/N_H^{r+1}H^{\bullet}(X,\mathbb{Q})$$
is generated in degrees at most $l+r$ for all $r< c$.
\end{proposition}
\begin{proof}
Observe that the subspace of $H^d(X,\mathbb{Q})$ generated in degrees at most $l+r$ is a sub-Hodge structure whose complexification is the subspace of $H^d(X,\mathbb{C})$ generated in degrees at most $l+r$. Consider a simple Hodge structure 
$$V\subset N_H^rH^d(X,\mathbb{Q})$$
which is not contained in $N_H^{r+1}H^{d}(X,\mathbb{Q})$ and cycles $W_{1},\ldots, W_k\subset X\times X$ of codimension $d_1,\ldots, d_k \leq l$ with
\[n=_{\text{def}} \dim X= d_1+\cdots+d_k.\]
For each $i$, decompose the cycle class of $W_i$ into K\"unneth components
\[[W_i]=\sum_{a_i,b_i=0}^{d_i}v_i^{a_i,b_i}\otimes w_i^{d_i-a_i,d_i-b_i},\]
where 
\[v_i^{a_i,b_i}\otimes w_i^{d_i-a_i,d_i-b_i}\in H^{a_i,b_i}(X)\otimes H^{d_i-a_i,d_i-b_i}(X).\]
Given a non-zero class $\alpha\in V_\mathbb{C}^{r,d-r}$, we have
\[[W_1\cdots W_k]^*\alpha=\sum_{|\mathbf{a}|=r,|\mathbf{b}|=d-r}\left[\prod_{i=1}^k(v_i^{a_i,b_i}\otimes w_i^{d_i-a_i,d_i-b_i}) \right]^*\alpha \ \,\]
where the sum is over tuples $\mathbf{a}=(a_1,\ldots, a_k)$, $\mathbf{b}=(b_1,\ldots b_k)$ with sums $r$ and $d-r$ respectively. Since 
\[(v_1^{a_1,b_1}\otimes w_1^{d_1-a_1,d_1-b_1}\cdots v_k^{a_k,b_k}\otimes w_k^{d_k-a_k,d_k-b_k})^*\alpha\]
is a multiple of $v_1^{a_1,b_1}\cdots v_k^{a_k,b_k}$ and $a_i\leq r$, $b_i\leq d_i\leq l$, we must have $a_i+b_i\leq l+r$. This shows that $[W_1\cdots W_k]^*\alpha$ is in the subspace of $H^d(X,\mathbb{C})$ generated in degree $\leq l+r$, and thus that $V$ is contained in the subspace of $H^d(X,\mathbb{Q})$ generated in degree $\leq l+r$.
\end{proof}

\begin{proposition}\label{impliesgenbloch}
Let $X$ be a smooth projective $n$-fold with $H^{0}(X,\Omega^p)=0$ for all $p\geq 2$ and set $F^{2}\CH_0(X)$ to be either $F_{\mathrm{V}}^{2}\CH_0(X)$ or the kernel of the albanese map on zero-cycles of degree 0. 
If~$X$ has a degree $1$ polynomial decomposition of the diagonal up to coniveau $1$
then $F^{2}\CH_0(X)=0$. In particular, Conjecture \ref{polyconj} for $c=1$ implies the generalized Bloch conjecture for $c=1$. 
  \end{proposition}

\begin{proof}
Since $X$ has a degree $1$ polynomial decomposition of the diagonal up to coniveau $1$ we can write
$$\Delta_X=Z_1+Z_2\in \CH^n(X\times X),$$
where $Z_1$ is a polynomial in divisors and $Z_2$ is supported on $Y\times X$ for some divisors $Y\subset X$. Accordingly,
$$\text{id}_{\CH_0(X)}={\Delta_{X}}_*= {Z_1}_*: \CH_0(X)\longrightarrow \CH_0(X).$$
We will show that any monomial $D_1\cdots D_n$, where $D_i\in \text{Pic}(X\times X)$, gives a zero map on $F^2\CH_0(X)$ which will imply that $F^2\CH_0(X)=0$. Observe that
$$\textup{Pic}(X\times X)=\textup{pr}_1^*\textup{Pic}(X)\oplus \textup{pr}_2^*\textup{Pic}(X)\oplus H,$$
where $H$ is the group that consists of pullbacks of divisors classes on $\textup{Alb}(X)\times \textup{Alb}(X)$ which have trivial restriction to the factors.
\medskip

Given a monomial $D_1\cdots D_n$, if one of the $D_i$ is in $\text{pr}_1^*\textup{Pic}(X)$ then the map
$$(D_1\cdots D_n)_*: \CH_0(X)\longrightarrow \CH_0(X)$$
is identically zero. Similarly, if at least $n-1$ of the $D_i$ belong to $\text{pr}_2^*\text{Pic}(X)$ then the same map factors through the Chow group of zero-cycles on a variety of dimension $1$ so that the induced map
$$(D_1\cdots D_n)_*: F^2\CH_0(X)\longrightarrow F^2\CH_0(X)$$
is identically zero.
\medskip

Finally, suppose that $D_1,\ldots, D_i$ belong to $\text{pr}_2^*\text{Pic}(X)$ and that $D_{i+1},\cdots, D_{n}$ are pulled back from divisors $D_{i+1}',\ldots, D_{n}'$ on $\text{Alb}(X)\times\text{Alb}(X)$ under $\alpha\times \alpha$, where $\alpha: X\longrightarrow \text{Alb}(X)$ is the Albanese morphism. Since $H^0(X,\Omega^2)=0$, the Albanese dimension of $X$ is at most $1$. We can assume that the image of $X$ in its Albanese is a curve $C$, or else $H=0$ and we are done by the considerations above. Since $D_{i+1}\cdot D_n=0$ if $i<n-2$, it suffices to consider the case $i=n-2$. Then $D_{n-1}\cdot D_n$ is the pullback of a zero cycle on $C\times C$. Since the pullback of a point on $C\times C$ is the product of a divisor in $\textup{pr}_1^*\textup{Pic}(X)$ and a divisor in $\textup{pr}_2^*\textup{Pic}(X)$, we see that $D_1\cdots D_n$ induces the zero map on $\CH_0(X)$.
\end{proof}

\subsection{Relation between Conjecture \ref{polyconj} and Main Conjecture \ref{mainconj}}
The following proposition shows that Conjecture \ref{polyconj} for $c=1$ along with Nori's Conjecture \ref{noriconj} implies
Main Conjecture~\ref{mainconj} using Voisin's filtration as a candidate Bloch--Beilinson filtration.

\begin{proposition}\label{P:polNorimain}
	Assume that $X$ has a degree $l$ polynomial decomposition of the diagonal and that $X$ satisfies Nori's Conjecture~\ref{noriconj}. 
	Then, for closed point $x_1,\ldots, x_m,y_1,\ldots, y_m\in X,$
	\begin{equation*}
	\sum_{i=1}^m [x_i]=\sum_{i=1}^m [y_i] \ \mbox{in} \ \CH_0(X) \ \ \iff \ \ \sum_{i=1}^m [x_i]=\sum_{i=1}^m [y_i] \ \mbox{in} \ \CH_0(X)/F_{\mathrm{V}}^{ml+1}\CH_0(X).
	\end{equation*} 
\end{proposition}

\begin{proof}
By assumption
 $\Delta_X=Z_1+Z_2$, where $Z_1$ is in the subalgebra of $\textup{CH}^\bullet(X\times X)$ generated in degrees $\leq l$, and $Z_2$ is supported on $Y\times X$ for some proper closed subset $Y\subset X$.
 Hence,
$${\Delta_X}_*={Z_1}_*: \textup{CH}_0(X)\longrightarrow \textup{CH}_0(X).$$
Consider cycles $W_1,\ldots, W_n$ in $\bigoplus_{i\leq l}\textup{CH}^{i}(X\times X)$ generating a subalgebra containing $Z_1$.
Consider $\mathbf{a}=(a_1,\ldots, a_n)\in \mathbb{N}^n\setminus \{\mathbf{0}\}$ with $a_1+\cdots+a_n\leq m$. The cycle
$$P_{\mathbf{a}}=_{\textup{def}}W_1^{a_1}\cdots W_n^{a_n}$$
has codimension at most $ml$. Given $x_1,\ldots, x_m, y_1,\ldots, y_m\in X$ such that
$$\sum_{i=1}^m [x_i]-\sum_{i=1}^m [y_i]\ \ \mbox{lies in}\ F^{ml+1}_{\mathrm{V}}\textup{CH}_0(X),$$ 
Nori's Conjecture \ref{noriconj} and Lemma \ref{corresp} imply that
$$\sum_{i=1}^m P_{\mathbf{a},x_i}=\sum_{i=1}^m P_{\mathbf{a},y_i}.$$
The same argument as in the proof of Proposition \ref{HKprop} shows that
$$P_{\mathbf{a},x_i}= W_{1,x_i}^{a_1}\cdots W_{n,x_i}^{a_n}\in \textup{CH}^\bullet(X),$$
so that
\begin{equation} \label{equalpa}\sum_{i=1}^m W_{1,x_i}^{a_1}\cdots W_{n,x_i}^{a_n}=\sum_{i=1}^m W_{1,y_i}^{a_1}\cdots W_{n,y_i}^{a_n}\in \textup{CH}^\bullet(X).\end{equation}

In order to proceed, we will need a generalization of the fundamental theorem on symmetric polynomials. 
Consider variables $w_j^{(i)}$ where $1\leq j\leq n$ and $1\leq i\leq m$. 
The symmetric group $\mathfrak{S}_m$ acts on the ring $\mathbb{Q}[w_{j}^{(i)}]$ by permuting the superscript and the subalgebra $\mathbb{Q}[w_{j}^{(i)}]^{\mathfrak{S}_m}$ is called the subalgebra of multi-symmetric polynomials. 
Given an element $\mathbf{a}=_{\textup{def}}(a_1,\ldots, a_n)\in \mathbb{N}^n\setminus\mathbf{0}$, the corresponding power sum multisymmetric polynomial is
$$p_{\mathbf{a}}(w_{j}^{(i)})=\sum_{s=1}^m {w_1^{(s)}}^{a_1}\cdots {w_n^{(s)}}^{a_n}\in \mathbb{Q}[w_{j}^{(i)}]^{\mathfrak{S}_m}.$$ It is a classical fact that $\mathbb{Q}[w_{j}^{(i)}]^{\mathfrak{S}_m}$ is generated by elementary multisymmetric power sums (see the references in the introduction of \cite{Briand}). 
\begin{prop}[{\cite[Cor.~5]{Briand}}]\label{multisym}
All power sums are in the ideal generated by the powers sums $p_{\mathbf{a}}(w_{j}^{(i)})$ with $a_1+\cdots+a_n\leq m$.
\end{prop}

Now consider two morphisms
$f,g: \mathbb{Q}[w_{j}^{(i)}]\longrightarrow \CH^\bullet(X)$ given respectively by $f(w_j^{(i)})=W_{j, x_i}$ and $g(w_j^{(i)})=W_{j, y_i}$. The equality \eqref{equalpa} amounts to $f(p_{\mathbf{a}}(w_{j}^{(i)}))=g(p_{\mathbf{a}}(w_{j}^{(i)}))$ for all $\mathbf{a}=(a_1,\ldots, a_n)$ in $\mathbb{N}^n\setminus \{\mathbf{0}\}$ satisfying $a_1+\cdots+a_n\leq m$.
Proposition \ref{multisym} then implies that $f(p_{\mathbf{a}}(w_{j}^{(i)}))=g(p_{\mathbf{a}}(w_{j}^{(i)}))$, namely
$$\sum_{i=1}^m P_{\mathbf{a},x_i}=\sum_{i=1}^m P_{\mathbf{a},y_i}\in \textup{CH}^\bullet(X),$$
for any $\mathbf{a}=(a_1,\ldots, a_n)\in \mathbb{N}^n\setminus \{\mathbf{0}\}$. Since $Z_1$ is a polynomial in the $W_i$, it follows that
$$\sum_{i=1}^m [x_i]=\sum_{i=1}^m Z_{1,x_i}=\sum_{i=1}^m Z_{1,y_i}=\sum_{i=1}^m [y_i]\ \ \mbox{in}\ \textup{CH}_0(X),$$
thereby concluding the proof of the proposition.
\end{proof}

\vspace{1em}

\bibliography{biblio}

\begin{thebibliography}{BFMS22}

\bibitem[Bea86]{Beau2}
A.~Beauville.
\newblock Sur l'anneau de {Chow} d'une vari{\'e}t{\'e} ab{\'e}lienne.
\newblock {\em Math. Ann.}, 273:647--651, 1986.

\bibitem[BFMS22]{BFMS}
Ignacio Barros, Laure Flapan, Alina Marian, and Rob Silversmith.
\newblock On product identities and the {C}how rings of holomorphic symplectic
  varieties.
\newblock {\em Selecta Math. (N.S.)}, 28(2):Paper No. 46, 29, 2022.

\bibitem[Bri04]{Briand}
Emmanuel Briand.
\newblock When is the algebra of multisymmetric polynomials generated by the
  elementary multisymmetric polynomials?
\newblock {\em Beitr{\"a}ge Algebra Geom.}, 45(2):353--368, 2004.

\bibitem[BS83]{BlochSrinivas}
S.~Bloch and V.~Srinivas.
\newblock Remarks on correspondences and algebraic cycles.
\newblock {\em Amer. J. Math.}, 105 (5):1235--1253, 1983.

\bibitem[Flo23]{Floccari}
Salvatore Floccari.
\newblock {On the motive of O'Grady's six dimensional hyper-K\"{a}hler
  varieties}.
\newblock {\em {\'Epijournal de G\'eom\'etrie Alg\'ebrique}}, {Volume 7},
  February 2023.

\bibitem[FLV19]{FLV-Franchetta}
Lie Fu, Robert Laterveer, and Charles Vial.
\newblock The generalized {F}ranchetta conjecture for some hyper-{K}\"{a}hler
  varieties.
\newblock {\em J. Math. Pures Appl. (9)}, 130:1--35, 2019.
\newblock With an appendix by the authors and Mingmin Shen.

\bibitem[FTV19]{ftv}
Lie Fu, Zhiyu Tian, and Charles Vial.
\newblock Motivic hyper-{K}\"{a}hler resolution conjecture, {I}: generalized
  {K}ummer varieties.
\newblock {\em Geom. Topol.}, 23(1):427--492, 2019.

\bibitem[Jos95]{J}
K.~Joshi.
\newblock A {N}oether-{L}efschetz theorem and applications.
\newblock {\em J. Algebraic Geom.}, 4:105--135, 1995.

\bibitem[KMP07]{KMP}
Bruno Kahn, Jacob~P. Murre, and Claudio Pedrini.
\newblock On the transcendental part of the motive of a surface.
\newblock In {\em Algebraic cycles and motives. {V}ol. 2}, volume 344 of {\em
  London Math. Soc. Lecture Note Ser.}, pages 143--202. Cambridge Univ. Press,
  Cambridge, 2007.

\bibitem[KS16]{KS}
B.~Kahn and S.~Sujatha.
\newblock {Birational motives, I: pure birational motives}.
\newblock {\em Ann. K-Theory}, 1(14):379--440, 2016.

\bibitem[K{\"u}n94]{K-abelian}
Klaus K{\"u}nnemann.
\newblock On the {C}how motive of an abelian scheme.
\newblock In {\em Motives ({S}eattle, {WA}, 1991)}, volume~55 of {\em Proc.
  Sympos. Pure Math.}, pages 189--205. Amer. Math. Soc., Providence, RI, 1994.

\bibitem[Lin18]{lin-corr}
Hsueh-Yung Lin.
\newblock Corrigendum to ``{O}n the {C}how group of zero-cycles of a
  generalized {K}ummer variety'' [{A}dv. {M}ath. 298 (2016) 448--472] [
  {MR}3505747].
\newblock {\em Adv. Math.}, 331:1016--1021, 2018.

\bibitem[MRS18]{MRS}
Giovanni Mongardi, Antonio Rapagnetta, and Giulia Sacc\`a.
\newblock The {H}odge diamond of {O}'{G}rady's six-dimensional example.
\newblock {\em Compos. Math.}, 154(5):984--1013, 2018.

\bibitem[Mur90]{Murre}
J.P. Murre.
\newblock On the motive of an algebraic surface.
\newblock {\em J. Reine Angew. Math.}, 409:190--204, 1990.

\bibitem[Mur93]{Murre2}
J.~P. Murre.
\newblock On a conjectural filtration on the {C}how groups of an algebraic
  variety. {I}. {T}he general conjectures and some examples.
\newblock {\em Indag. Math. (N.S.)}, 4(2):177--188, 1993.

\bibitem[MZ20]{MZ}
A.~Marian and X.~Zhao.
\newblock On the group of zero-cycles of holomorphic symplectic varieties.
\newblock {\em {\'E}pijournal G{\'e}om. Alg{\'e}brique}, 4, 2020.

\bibitem[Nor93]{N}
M.~V. Nori.
\newblock Algebraic cycles and {Hodge} theoretic connectivity.
\newblock {\em Invent. Math.}, 111(2):349--374, 1993.

\bibitem[O'G03]{OG6}
Kieran~G. O'Grady.
\newblock A new six-dimensional irreducible symplectic variety.
\newblock {\em J. Algebraic Geom.}, 12(3):435--505, 2003.

\bibitem[O'G13]{OG-moduli}
Kieran~G. O'Grady.
\newblock Moduli of sheaves and the {C}how group of {$K3$} surfaces.
\newblock {\em J. Math. Pures Appl. (9)}, 100(5):701--718, 2013.

\bibitem[SV16]{SV}
Mingmin Shen and Charles Vial.
\newblock The {F}ourier transform for certain hyperk\"{a}hler fourfolds.
\newblock {\em Mem. Amer. Math. Soc.}, 240(1139):vii+163, 2016.

\bibitem[SYZ20]{SYZ}
Junliang Shen, Qizheng Yin, and Xiaolei Zhao.
\newblock Derived categories of {$K3$} surfaces, {O}'{G}rady's filtration, and
  zero-cycles on holomorphic symplectic varieties.
\newblock {\em Compos. Math.}, 156(1):179--197, 2020.

\bibitem[Via22]{Vial-JMPA}
Charles Vial.
\newblock On the birational motive of hyper-{K}{\"a}hler varieties.
\newblock {\em J. Math. Pures Appl.}, 163:577--624, 2022.

\bibitem[Voi02]{VoisinHodge}
Claire Voisin.
\newblock {\em Th{\'e}orie de {Hodge} et g{\'e}om{\'e}trie alg{\'e}brique
  complexe}, volume~10 of {\em Cours Sp{\'e}cialis{\'e}s}.
\newblock Soci{\'e}t{\'e} Math{\'e}matique de France, 2002.

\bibitem[Voi04]{V4}
Claire Voisin.
\newblock Remarks on filtrations on {C}how groups and the {B}loch conjecture.
\newblock {\em Ann. Mat. Pura Appl. (4)}, 183:421--438, 2004.

\bibitem[Voi16]{V2}
Claire Voisin.
\newblock Remarks and questions on coisotropic subvarieties and 0-cycles of
  hyper-{K{\"a}hler} varieties.
\newblock In {\em K3 Surfaces and Their Moduli, Proceedings of the
  Schiermonnikoog conference 2014}, volume 315 of {\em Progress in Math 315},
  pages 365--399, 2016.

\bibitem[Voi22]{V3}
Claire Voisin.
\newblock On the {L}efschetz standard conjecture for {L}agrangian covered
  hyper-{K}\"{a}hler varieties.
\newblock {\em Adv. Math.}, 396, 2022.

\end{thebibliography}

\end{document}